\documentclass[11pt]{amsart}

\usepackage{amsthm}
\usepackage{amssymb}   
\usepackage{amsmath}  

\newtheorem{thm}{Theorem}[section]
\newtheorem{lem}[thm]{Lemma} 
\newtheorem{cor}[thm]{Corollary} 
\newtheorem{Def}[thm]{Definition}
\newenvironment{dfn}{\begin{Def} \rm}{\end{Def}}

\usepackage{tikz}

\begin{document}

\title[Terwilliger algebras of  Bol loops]{Terwilliger algebras constructed from Cayley tables of finite  Bol loops}

\author{Brian Curtin}
\address{
Department of Mathematics and Statistics,
University of South Florida,
Tampa {FL} {33620} {USA}}
\email{bcurtin@usf.edu
}

\begin{abstract}  
 We describe the Terwilliger algebras of the four-class Latin-square association schemes arising from Cayley tables of Bol loops.  
 We give some necessary conditions involving Terwilliger algebras for a quasigroup to be a Bol loop.
\end{abstract}


\maketitle

\section{Introduction}

In this paper, we examine Terwilliger algebras constructed from finite Bol loops.    Bol loops are quasigroups, and the Cayley table of a finite quasigroup is a Latin square \cite{Pflugfleder:QGLI,Smith:IQGR,SmithRomanowsak:PMA}.  From a Latin square, one may construct a four-class association scheme.  The Bose-Mesner algebra of an association scheme encodes the structure constants of the association scheme (Section \ref{sec:T-latinI}).  Terwilliger algebras supplement the Bose-Mesner algebra with additional combinatorial data by fixing a base point and encoding data concerning subconstituents of the association scheme relative to the base point.   Any small insight into  not-necessarily associative finite quasigroups and Bol loops  is secondary to our goal of understanding the Terwilliger algebras constructed from Bol loops.

We build upon the description of Terwilliger algebras of Latin squares given in \cite{CurtinDaqqa:SALS,CurtinDaqqa:SSRGLS} (Section \ref{sec:T-latinII}) to determine those of Cayley tables Bol loops.   We find that Bol loops have particularly nice Terwilliger algebras, even among quasigroups.  Namely, depending upon the properties of the base point of the Terwilliger algebra has either 3 or 4 isomorphism classes of irreducible modules.  We  also describe some properties of Terwilliger algebras of a Bol loop (Section \ref{sec:Tbol}).   We conclude with a  sufficient condition for a quasigroup to be a Bol loop which partially involves a condition related to Terwilliger algebras  (Section \ref{sec:BolfromT}).

\section{Terwilliger algebras of Latin squares, I}\label{sec:T-latinI}

We recall association schemes and Terwilliger algebras of Latin squares. 

\begin{dfn}       
\label{def:Lparts}
Let $n$ be a positive integer. A {\em Latin square of order $n$}
is an $n\times n$ array $L$ filled with $n$ distinct symbols such that 
every symbol occurs  exactly once in each row and in each column of $L$.
Let $I(L)$ denote the index set for the rows and columns of $L$.   
Let $X(L)=\{(r,c,L(r,c))\,|\, r, c\in I(L)\}$ be the associated orthogonal array.
\end{dfn}       

See \cite{DenesKeedwell:LSA,DenesKeedwell:LSND,LaywineMullen:DMLS} for more on Latin
squares.  The orthogonal array $X(L)$ encodes $L$, and any two components of a point in $X(L)$ uniquely determine the third.

\begin{dfn}       
\label{def:L->AS}
Let $L$ be a Latin square of order $n$, and  let $\mathbb{M}_{X}$ denote the set of complex matrices whose rows and columns are indexed by $X=X(L)$.
Define relation matrices $\{A_i\}_{i=0}^4$ in $\mathbb{M}_{X}$ as follows. For all $x$, $x'\in X$, the respective $(x, x')$-entry of $A_i$ $(i=0,1,2,3,4)$ is 1 if ($i=0$) $x$ and $x'$ agree in all three components, ($i=1$) in only the first component, ($i=2$) in only the second component, ($i=3$) in only the third component, and ($i=4$) in no component; all other entries of these matrices are zero.  
The {\em association scheme of $L$} is the pair $\mathcal{A}_L=(X,  \{A_i\}_{i=0}^4)$.
\end{dfn}       

\begin{dfn}      
Let $L$ be a Latin square, and let $\mathcal{A}_L=(X,  \{A_i\}_{i=0}^4)$ be the association scheme of $L$.  The {\em Bose-Mesner algebra} of $L$ and $\mathcal{A}_L$ is the $\mathbb{C}$-subalgebra of $\mathbb{M}_{X}$ generated by $\{A_i\}_{i=0}^4$.
\end{dfn}       

Association schemes of Latin squares and their Bose-Mesner algebras are discussed in \cite{bailey:ASDEAC,Delsarte:AAASCT}.  General references for association schemes include \cite{bailey:ASDEAC,BannaiIto:ACI,BrouwerCohenNeumair:DRG}.   
Association schemes have structure constants, called {\em  intersection numbers},  $p^{h}_{ij}=|\{z\in X\,|\, A_{i}(x,z)A_{j}(z,y)\}|$ for $x$, $y\in X$ with $A_h(x,y)=1$.  Since  $A_iA_j=\sum_{h}p^{h}_{ij}A_h$, the isomorphism class of a Bose-Mesner algebra depends only upon the $p^{h}_{ij}$.  For  Latin squares, the $p^{h}_{ij}$ depend only on the order of the Latin square.

\begin{dfn}       
\label{def:Talg}
Let $L$ be a Latin square.  Fix $p=(r_{p}, c_{p}, L(r_{p} c_{p}))\in X(L)$ (the {\em base point}).  Define diagonal matrices (the {\em dual idempotents}) $\{E^*_i\}_{i=0}^4$ in $\mathbb{M}_{X}$ as follows:  For all $x\in X(L)$, the $(x,x)$-entry of $E^*_i$ equals the $(p,x)$-entry of $A_i$ $(i=0,1,2,3,4)$.  The {\em Terwilliger} or {\em subconstituent algebra} $\mathcal{A}_L(p)$ of $\mathcal{A}_L$ with respect to $p$ is the subalgebra of $\mathbb{M}_{X}$ generated by $\{A_i\}_{i=0}^4\cup\{E^{*}_i\}_{i=0}^4$.  
\end{dfn}       

Terwilliger algebras are associative  finite-dimensional semisimple complex algebras, so   by Wedderburn theory, each is isomorphic to a direct sum of full complex matrix algebras.   The matrix summands are in bijective correspondence with the isomorphism classes of irreducible $\mathcal{A}_L(p)$-modules, where the size of the summand is the dimension of each module in the class.  Since $\mathcal{A}_L(p)$ acts on  $\mathbb{C}^{n^2}$ by left multiplication,  $\mathbb{C}^{n^2}$ decomposes into the direct sum of irreducible $\mathcal{A}_L(p)$-modules. The {\em multiplicity} of an irreducible $\mathcal{A}_L(p)$-modules is the number of isomorphic copies that appear in any decomposition of $\mathbb{C}^{n^2}$ into irreducible $\mathcal{A}_L(p)$-modules.  See the recent survey \cite{Terwilliger:DRGSAQPP} or the original \cite{Terwilliger:SAI} for more details.   

Terwilliger algebras were first introduced to study P- and Q-polynomial association schemes \cite{Terwilliger:SAI} and their connection to the orthogonal polynomials in the Askey scheme. Since their introduction, the Terwilliger algebras of various families of association schemes have been examined, and several applications have been developed.  

\section{Main class equivalence}
\label{sec:mainclass}

We recall some combinatorial equivalences of Latin squares and the relationship between Terwilliger algebras of equivalent Latin squares.

An {\em isotopy} of a Latin square $L$ is a triple of permutations 
$\Sigma=(\sigma_r, \sigma_c, \sigma_e)$, where the rows, columns, and entries of $L$  are permuted by $\sigma_r$, $\sigma_c$, and $\sigma_e$, respectively. 
Latin squares $L$ and $L'$  are said to be {\em isotopic} if $L'=\Sigma(L)$ for some isotopy $\Sigma$. 

\begin{lem}       
\cite[Thm.~5.4.2]{Daqqa:SALS}
\label{lem:isotopyisom}
Let $L$  and $L'$ be  isotopic Latin squares via isotopy $\Sigma$. Then  $\mathcal{A}_L(p)$ is  isomorphic to $\mathcal{A}_{L'}(\Sigma(p))$ for all $p\in X(L)$.
\end{lem}       

Isotopy does not respect quasigroup structure. We may consider only loops
since every Latin square is isotopic to the Cayley table of a loop.   

By a {\em conjugacy} of a Latin square $L$, we mean a permutation $\sigma\in\mathrm{Sym}(3)$  that permutes the components of elements of $X(L)$ according to $\sigma$ (the entry in component $i$ is moved to component $\sigma(i)$).     Latin squares $L$ and $L'$  are said to be {\em conjugate} when
$X(L')=\sigma(X(L))$ for a conjugacy $\sigma$.

\begin{lem}       
\cite[Thm.~5.4.2]{Daqqa:SALS}
\label{lem:conjugacyisom}
Let $L$  and $L'$ be conjugate Latin squares via conjugacy $\sigma$.  Then  $\mathcal{A}_L(p)$ is isomorphic to $\mathcal{A}_{L'}(\sigma(p))$ for all  $p\in X(L)$.
\end{lem}       

Transposing a Latin square corresponds to swapping the first two coordinates in $X(L)$.  Transposition maps Cayley tables of left Bol loops to those of right Bol loops, so we may focus on Terwilliger algebras of right Bol loops.  Conjugacy need not respect quasigroup structure.

Latin squares $L$ and $L'$  are said to be {\em main class equivalent} when $L$ is isotopic to a conjugate of $L'$.  Main class equivalent Latin squares (Cayley tables of quasigroups) have  isomorphic Terwilliger algebras for corresponding base points.

\section{Terwilliger algebras of Latin squares, II}\label{sec:T-latinII}

The structure of a  Terwilliger algebra of a Latin square depends upon the Latin square and the base point  via the following  permutation \cite{CurtinDaqqa:SALS, Daqqa:SALS}.  

\begin{dfn}       
\label{def:permutation}
Let $L$ be a Latin square of order $n$. Pick $p=(r_p, c_p, e_p)\in X(L)$. Define a permutation  $\pi=\pi_{L,p}$  of $I(L) - \{c_p\}$ as follows: For  $c\in I(L) - \{c_p\}$, successively define auxiliary points $x$, $y$, $z\in X$ by  $x=(r_{p},c, L(r_p,c))$,  $y=(r,c_{p}, L(r_p,c))$, and $z=(r, c', e_{p})$.  Define $\pi(c)=c'$.  Note: $r$ and $c'$ are uniquely determined by the known entries in $y$ and $z$, respectively. 
\end{dfn}       
 
\begin{figure}[h]       
\begin{tabular}{c|c}
\begin{tikzpicture}  
\matrix [column sep=.5cm,row sep={.65cm,between origins}]
 {
\node[draw] (11) {p=1,1,1}; & \node (12) {x=1,2,2\phantom{'}};	& \node (13) {x'=1,3,3};\\
\node (21)	{y=2,1,2}; &         & \node (23) {z=2,3,1\phantom{'}}; \\ 
\node(31) {y'=3,1,3}; & \node (32) {z'=3,2,1}; &  
                                                                                                    \\
}; 
\draw [->, thick] (12.west) -- (21.east);
\draw [->, thick] (21.east) -- (23.west);
\draw [->, thick] (23.west) -- (13.west);

\draw [->, dashed] (13.west) -- (31.east);
\draw [->, dashed] (31.east) -- (32.west);
\draw [->, dashed] (32.west) -- (12.west);
\end{tikzpicture} 
&
\begin{tikzpicture} 
\matrix [column sep=1cm,row sep={.65cm,between origins}]
 {
\node[draw] (11) {1}; & \node (12) {2};	& \node (13) {3}; \\
\node (21)	{2}; & \node (22) {3}; & \node (23) {1};  \\ 
\node(31) {3}; & \node (32) {1}; & \node (33) {2};   \\
}; 
\draw [->, thick] (12.west) -- (21.east);
\draw [->, thick] (21.east) -- (23.west);
\draw [->, thick] (23.west) -- (13.west);

\draw [->, dashed] (13.west) -- (31.east);
\draw [->, dashed] (31.east) -- (32.west);
\draw [->, dashed] (32.west) -- (12.west);

\end{tikzpicture}
\\%
 Auxiliary points in $X(L)$ &  Latin square $L$ 
\end{tabular}
\caption{The action of $\pi$ on  $L$ and $X(L)$.}\label{fig:computepi}
\end{figure}
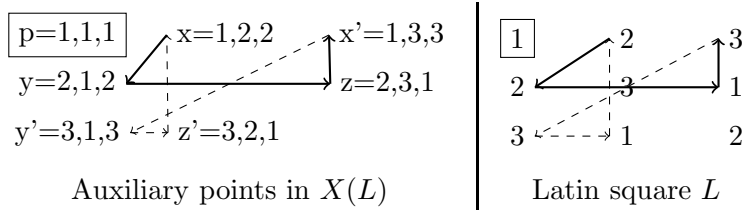       

 Figure \ref{fig:computepi} illustrates  Definition \ref{def:permutation}.  Base point  $p=(1,1,1)$ is marked, and  $\pi$ applied to column $c=2$. Now $x=(1,2,2)$ (column $c$, row of  $p$), $y=(2,1,2)$ (column of $p$, entry of $x$), and $z=(2,3,1)$ (row of $y$, entry of $p$).  Thus $\pi(2)=3$, the column of $z$.  Continuing, $x'=(1,3,3)$ (column of $z$,  row of $p$), $y'=(1,3,3)$, and $z'=(3,2,1)$.  Next $x''=(1,2,2)=x$.  Thus $\pi^2(2)=2$, and we get cycle (23) of $\pi$ consisting of  columns in the order in which they are encountered.   Note $\pi$ can be read from $L$ directly.

We recall the dimensions and multiplicities of the irreducible $\mathcal{A}_L(p)$-modules.  The complete action can be found in \cite{CurtinDaqqa:SALS}; we omit the details.

\begin{thm}        
\label{thm:tmods}
\cite{CurtinDaqqa:SALS}
With the reference to Definition \ref{def:permutation},  assume $n\geq 4$.
Write $\pi_{L,p}={\mathcal C}_1{\mathcal C}_2\cdots {\mathcal C}_k$ as a product of disjoint cycles, including 1-cycles.  For $1\leq i \leq k$, write $u_i=\{\epsilon\in\mathbb{C}\,|\, \epsilon^{|{\mathcal C}_i|}\}$, and let $U=\cup_{i=0}^{k} u_i$. 
The isomorphism classes of irreducible $\mathcal{A}_L(p)$-modules consist of 
(i) one of dimension $5$ and multiplicity 1 (the {\em primary module});
(ii) one of dimension $1$ and multiplicity $n^2-6n+7$;
(iii) one of dimension 6  and multiplicity $|\{ i\,|\, \epsilon\in u_i\}|$ for each $\epsilon\in U\backslash\{1\}$;
(iv) one of dimension 6 and multiplicity $k-1$ if $k>1$;
 and no others.
\end{thm}       

Theorem \ref{thm:tmods} and the discussion after Definition \ref{def:Talg} give the following result.

\begin{thm}       
\label{thm:permtoisom}
With reference to Theorem \ref{thm:tmods},  let $N= |U|-1$ if $k=1$, and $N=|U|$ if $k>1$.  Then
$\mathcal{A}_L(p)\cong{\mathbb{M}}_{5}\oplus{\mathbb{M}}_{6}^{\oplus N}\oplus{\mathbb{M}}_{1}$, where ${\mathbb{M}}_{\ell}$ is the complex algebra of all complex $(\ell\times \ell)$-matrices.
\end{thm}       

\section{Loops}\label{sec:T-loop}

We recall a few facts about quasigroups, loops, and  Bol loops.

A {\em quasigroup} is a set $Q$ with a binary operation (denoted by juxtaposition) such that for all  $a$, $b\in Q$,  there exist unique $c$, $d\in Q$ such that $ac=b$ and $da=b$.

Cayley tables of  finite quasigroups coincide with  Latin squares.
Let $L$ be the Cayley table of a finite quasigroup $Q$. The rows and columns of $L$ are indexed by $I(L)=Q$,  the $(a,b)$-entry of $L$ is  $ab$ $(a,b\in Q)$, and
 $X(L)=\{(a, b, ab)\,|\, a,b\in Q\}$. 
See \cite{Pflugfleder:QGLI,SmithRomanowsak:PMA,Smith:IQGR} for basic properties of quasigroups.
We will use the cancellation property of quasigroups. 
\begin{equation}
\label{eq:uniqueleftdivisor}
\hbox{If $ac=ad$, then $c=d$} \hbox{ for all $a,\, c,\, d\in Q$}.
\end{equation}

 A {\em loop} is a quasigroup $Q$ with an identity element $\iota\in Q$: For all $a\in Q$, $a\iota=a=\iota a$.     
The identity $\iota$ of a loop is unique, and each element $a$ of a loop has a unique left inverse $a^\lambda$ and a unique right inverse $a^\rho$.  
Note that $a$ has a two-sided inverse when $a^\lambda=a^\rho$, in which case we denote it $a^{-1}$.   
See \cite{Pflugfleder:QGLI,SmithRomanowsak:PMA} for further discussion of loops. 

A  {\em right Bol  loop} is a loop $Q$ that satisfies the right Bol axiom
\begin{equation}
\label{eq:Rightbolaxiom}
((ca)b)a = c((ab)a)   \hbox{ for all $a$, $b$, $c\in Q$}.
\end{equation}
A {\em left Bol loop} is a loop $Q$ that satisfies 
$a(b(ac)) = (a(ba))c$  for all $a$, $b$, $c\in Q$. 
A {\em Bol loop} is a loop that is either right or left Bol.
See \cite{Pflugfleder:QGLI,Robinson:BL} for basic properties. 
Every element of a Bol loop has a  two-sided inverse.
A right Bol loop $Q$ has the right inverse property: 
\begin{equation}
\label{eq:BolRIP}
 (ab)b^{-1}= a   \hbox{ for all $a$, $b\in Q$},
 \end{equation}
 and similarly a left Bol loop has the analogous left inverse property.

A {\em Moufang loop} is a loop that  is both  left and  right Bol.  
Every group is a Moufang loop since groups are associative.
A Moufang loop $Q$  has the antiautomorphic inverse property:  
 \begin{equation}\label{eq:moufantiautinv}
 (ab)^{-1}=b^{-1}a^{-1}\hbox{ for all $a$, $b\in Q$}. 
\end{equation}

\section{Terwilliger algebras of Bol loops}\label{sec:Tbol}

We describe the Terwilliger algebras of Bol loops. 
   
\begin{lem}       
\label{lem:looppidef}
Let $L$ be the Cayley table of a finite quasigroup $Q$.    Fix  $p=(r_{p}, c_{p}, r_{p}c_{p})\in X(L)$, and let $\pi=\pi_{L,p}$ be as in  Definition \ref{def:permutation}. 
Then for all $c\in Q-\{c_p\}$, there exists $r\in Q$ such that $\pi(c)$ and $r$ satisfy
\begin{eqnarray}\label{eqn:xiyi}
r c_{p}= r_{p} c, \label{eqn:xiyi-a}\\
r\pi(c)=r_{p}c_{p}.\label{eqn:xiyi-b}
\end{eqnarray}
\end{lem}       

\begin{proof}       
Definition \ref{def:permutation} gives auxiliary points
 $x=(r_{p}, c, r_{p}c)$,   $y=(r, c_{p}, r_{p}c)$, and $z=(r, \pi(c), r \pi(c))$.  The third entry of $y$ is both that of $x$ (namely $r_{p}c$) and $r c_{p}$.  The third entry of $z$ is both that of $p$ (namely $r_{p}c_{p}$) and $r\pi(c)$.
\end{proof}   

\begin{thm}        
\label{thm:bolorder}
Let $L$ be the Cayley table of a  finite right Bol loop $Q$.  Fix $p=(r_{p},c_{p},r_{p}c_{p})\in X(L)$, and let $\pi=\pi_{L,p}$ be as in Definition \ref{def:permutation}.  Then $\pi^2= \mathrm{id}$ and $\pi(c) = (c_p c^{-1})c_p$ for all $c\in Q-\{c_{p}\}$.  Moreover, $\pi(c)=c$ if and only if $cc_{p}^{-1}=c_{p}c^{-1}$.
\end{thm}       

\begin{proof}       
 We write $c_i=\pi^{i-1}(c)$, and let $r_i$ be such that $r_ic_p= r_pc_i$.  Compute. 
 
 \begin{alignat}{3}
&&&\text{Solve (\ref{eqn:xiyi-b}) for $r_1$ with (\ref{eq:BolRIP}):} & 
   r_1&=(r_{p}c_{p})c_{2}^{-1}.  \label{eq:ri=} \\
&&&\text{Substitute (\ref{eq:ri=}) into (\ref{eqn:xiyi-a}):}  &      
((r_{p}c_{p})c_{2}^{-1})c_{p}   &=   r_{p}c_1.  \nonumber \\
&&&\text{Apply the right Bol axiom (\ref{eq:Rightbolaxiom}):}&
  r_{p}((c_{p}c_{2}^{-1})c_{p}) &=  r_{p}c_1. \nonumber \\
&&&\text{By uniqueness of divisors (\ref{eq:uniqueleftdivisor}):}  &    
  (c_{p}c_{2}^{-1})c_{p} &= c_1.  \label{eq:byb}\\
 &&&\text{Increase subscripts of (\ref{eq:byb}):}&
  (c_{p}c_{3}^{-1})c_{p}& = c_{2}. \label{eq:cpc3cp=c2}\\
&&&\text{Right multiply by $c_2^{-1}$ then $c_p$:}\quad&  
(((c_{p}c_{3}^{-1})c_{p})c_{2}^{-1})c_{p} 
            &= (c_{2}c_{2}^{-1})c_{p}.  \nonumber \\
&&& \text{Simplify the right side:}&         
         & = c_{p}. \nonumber\\
&&&\text{Apply the right Bol axiom (\ref{eq:Rightbolaxiom}):} \quad&        
 (c_{p}c_{3}^{-1})((c_{p}c_{2}^{-1})c_{p}) & =  c_p.   \nonumber \\
%
 &&&\text{Simplify the left side with (\ref{eq:byb}):}&
 (c_{p}c_{3}^{-1})c_1 &=  c_{p}.           \nonumber \\     
&&&\text{Use (\ref{eq:BolRIP}) to rewrite:}&
 c_{p}c_{3}^{-1} &= c_{p} c_{1}^{-1}.  \nonumber\\
 &&&\text{By cancellation (\ref{eq:uniqueleftdivisor}):}&
c_{3}^{-1} &= c_{1}^{-1}.  \nonumber \\
 &&&\text{By uniqueness of inverses:}& \qquad\quad
\pi^2(c_1)=c_{3}&=c_1. \nonumber 
\end{alignat} 
Now  $\pi(c_1) =c_2 = (c_p c_{1}^{-1})c_p$ by  (\ref{eq:cpc3cp=c2}), and  $c_1=(c_p \pi(c_1)^{-1})c_p$ by (\ref{eq:byb}).  Hence $\pi(c_1)=c_2$ if and only if $c_{1}c_{p}^{-1}=c_{p}c_{1}^{-1}$  by (\ref{eq:BolRIP}) and (\ref{eq:byb}).
\end{proof}       

We have the following when considering all the base points. 

\begin{cor}       
\label{cor:rowindep}
With reference to Theorem \ref{thm:bolorder}, the number of elements of $X(L)$ fixed by $\pi$ is independent of the first coordinate (row $r_{p}$) of $p$.
\end{cor}       

\begin{proof}       
The condition $cc_{p}^{-1}=c_{p}c^{-1}$ is independent of $r_{p}$.
\end{proof}      

Corollary \ref{cor:rowindep} gives row independence of the cycle structure of  $\pi_{L,p}$; however,  the cycle structure may depend upon the column of $p$.  Figure \ref{fig:rtbolex} contains the Cayley table $L$ of a right Bol loop.  For $p$ in  a given column, the cycle structure of $\pi_{L,p}$ is noted below the column. 

\begin{figure}[h]       
{
\[
\begin{array}{cccccccccc}
1& 2& 3& 4& 5& 6& 7& 8\\
2& 8& 6& 1& 7& 3& 5& 4\\
3& 7& 8& 6& 1& 4& 2& 5\\
4& 1& 7& 8& 6& 5& 3& 2\\
5& 6& 1& 7& 8& 2& 4& 3\\
6& 3& 4& 5& 2& 8& 1& 7\\
7& 5& 2& 3& 4& 1& 8& 6\\
8& 4& 5& 2& 3& 7& 6& 1\\
\hline\smallskip
12^3&1^32^2&1^32^2&1^32^2 & 1^32^2 &1^32 &1^42&12^3 
\end{array}
\]
}
\caption{A right Bol loop, with  cycle structure in each column }
\label{fig:rtbolex}
\end{figure}       

\begin{thm}       
\label{thm:bolisomclasses}
Let  $L$ be  the Cayley table of a finite Bol loop of order at least 4.  Then the Terwilliger algebra of $L$ with respect to any base point is isomorphic (as a complex associative algebra) to ${\mathbb{M}}_{5}\oplus{\mathbb{M}}_{6}\oplus{\mathbb{M}}_{1}$  or ${\mathbb{M}}_{5}\oplus{\mathbb{M}}_{6}\oplus{\mathbb{M}}_{6}\oplus{\mathbb{M}}_{1}$.
\end{thm}       

\begin{proof}       
Note that $M\geq 1$ since $n\geq 4$. By Theorem \ref{thm:permtoisom}, all cycles of $\pi$ have length at most  two  ($\pi^2=\mathrm{id}$) if and only if $M\leq 2$ (the only if direction requires $n\geq 4$).    By Theorem \ref{thm:bolorder},  $\pi^{2}(x)=x$ for all $x\in X(p)$ whenever $L$ is the Cayley table of a finite right Bol Loop.  Thus, the theorem holds.
\end{proof}       

Theorem \ref{thm:bolisomclasses} has a mild extension.

\begin{cor}       
\label{cor:MCbolisomclasses}
Let  $L$ be any Latin square which is main class equivalent to the Cayley table of a finite Bol loop of order at least 4.  Then the Terwilliger algebra of $L$ with respect to any base point is isomorphic (as a complex algebra) to ${\mathbb{M}}_{5}\oplus{\mathbb{M}}_{6}\oplus{\mathbb{M}}_{1}$  or ${\mathbb{M}}_{5}\oplus{\mathbb{M}}_{6}\oplus{\mathbb{M}}_{6}\oplus{\mathbb{M}}_{1}$. 
\end{cor}       

\begin{proof}       
Immediate from Lemmas \ref{lem:isotopyisom} and \ref{lem:conjugacyisom} and  Theorem \ref{thm:bolisomclasses}.
\end{proof}       

Each summand ${\mathbb{M}}_{6}$ in Theorem \ref{thm:bolisomclasses} and Corollary \ref{cor:MCbolisomclasses} is associated with a distinct eigenvalue (1 and $-1$) of the permutation matrix of $\pi$, as in  Theorem \ref{thm:tmods}.

\begin{cor}       
Let  $L$ be any Latin square which is main class equivalent to the Cayley table of a finite Bol loop of order at least 4. 
The multiplicity of irreducible $\mathcal{A}_L(p)$-modules associated with $-1$ is the number of 2-cycles in $\pi$ and the multiplicity of irreducible $\mathcal{A}_L(p)$-modules associated with $1$ is the number of cycles in $\pi$ minus one.
\end{cor}       

\begin{proof}       
Straightforward from Theorems \ref{thm:tmods} and  \ref{thm:permtoisom}.
\end{proof}        

\begin{cor}       
\label{cor:moufang}
Let $L$ be the Cayley table of a  finite Moufang loop $Q$. Fix $p=(r_{p},c_{p},r_{p}c_{p})\in X(L)$, and write $\pi=\pi_{L,p}$. Let $s$ denote the number elements of $Q$ which are their own inverse.
Then $\pi^2= \mathrm{id}$ and $\pi_{L,p}$ fixes exactly $s-1$ elements of $X(p)$.
\end{cor}       

\begin{proof}       
By Theorem \ref{thm:bolorder},  $c$ is fixed by $\pi=\pi_{L,p}$ 
if and only if $cc_{p}^{-1}=c_{p}c^{-1}$. However, by  (\ref{eq:moufantiautinv}),
$c_{p}c^{-1}= (cc_{p}^{-1})^{-1}$.  Thus $c$ is fixed by $\pi=\pi_{L,p}$ if and only if $cc_{p}^{-1}$ is its own inverse in $Q$. The case $c=c_{p}$ (when $cc_{p}^{-1}=\iota$) is excluded from the count. All other elements of $Q$ arise from the product $cc_{p}^{-1}$ as $c$ runs over $Q-\{c_{p}\}$. The result follows from Theorem  \ref{thm:bolorder}.
\end{proof}      

Corollary \ref{cor:moufang}  was shown for finite groups  in \cite{CurtinDaqqa:SALS}. 
The collective cycle structures of a quasigroup cannot distinguish groups and Moufang loops.

The Latin square $L$ in Figure \ref{fig:nonexample}  satisfies $\pi_{L,p}^2=\mathrm{id}$ for all base points $p$, but it is not main class equivalent to a Bol loop.  The cycle structure of $\pi_{L,p}$ is $2^3$  when $p$ is one of the boxed points  and $1^4 2$ otherwise.   The boxed points form a transversal, and its image under isotopy and conjugacy remains one.
Thus by Corollary \ref{cor:rowindep},  it is not main class equivalent to a Bol loop.  (Or, a Bol loop of prime order must be a group, and the cycle structure in this example is not that of the integers mod 7 by Corollary \ref{cor:moufang}.)

\begin{figure}[h]       
\[
\begin{array}{ccccccc}
 1 & 2 & 3 & 4 & \fbox{5} & 6 & 7 \\
 2 & 1 & 4 & 3 & 7 & 5 & \fbox{6} \\
 \fbox{3} & 7 & 5 & 6 & 1 & 4 & 2 \\
 4 & 3 & 2 & 1 & 6 & \fbox{7} & 5 \\
 5 & 6 & \fbox{1} & 7 & 3 & 2 & 4 \\
 6 & 5 & 7 & \fbox{2} & 4 & 1 & 3 \\
 7 & \fbox{4} & 6 & 5 & 2 & 3 & 1
\end{array}
\]
\caption{A non-example}
\label{fig:nonexample}
\end{figure}       

\section{Toward a converse}\label{sec:BolfromT}

We discuss when the Terwilliger algebra of a quasigroup  is that of a Bol loop.
We first interpret some cycle lengths.    

\begin{lem}       
\label{lem:iiicyclemean}
Let $L$ be the Cayley table of a finite loop $Q$ with identity $\iota$.  Fix base point $p=(\iota,\iota,\iota)$, and let $\pi=\pi_{L,p}$.  Fix $c\in Q-\{\iota\}$.  Then  $c$ is in a cycle of length $k$ if and only if $k$ is the least positive integer such that $c^\lambda = (((c^\rho)^\rho ) \cdots )^\rho$ with $(k-1)$-many $\rho$.
\end{lem}       

\begin{proof}       
Set $r_{p}=c_{p}=\iota$ in (\ref{eqn:xiyi-a}) and (\ref{eqn:xiyi-b}): $r=c$ and $rc'=\iota$.  Now $\pi(c)= c'=c^{\rho}$.   By induction $\pi^k(c) = (((c^\rho)^\rho ) \cdots )^\rho$, with $k$-many $\rho$.  The right side equals $c$ if and only if $1 = (((c^\rho)^\rho ) \cdots )^\rho c$ with $(k-1)$-many $\rho$, i.e.~$(((c^\rho)^\rho ) \cdots )^\rho = c^\lambda$.
\end{proof}       

\begin{cor}       
\label{cor:2sided}
With reference to Lemma \ref{lem:iiicyclemean},   $\pi^2(c)=c$ if and only if $c$ has a two-sided inverse. When these equivalent conditions hold, $\pi(c)=c^{-1}$.
\end{cor}       

\begin{proof}      
If $\pi^2(c)=c$, then $c^{\rho}=c^{\lambda}=c^{-1}$ by Lemma \ref{lem:iiicyclemean}.  Conversely, if $c$ has a two-sided inverse, then $c^{\rho}=c^{-1}$, and  $(c^{\rho})^\rho = c$, so $\pi^2(c)=c$.
\end{proof}       

\begin{cor}       
With reference to Lemma \ref{lem:iiicyclemean},
$\pi^2=\mathrm{id}$ if and only if every element of $Q$ has a two-sided inverse.
\end{cor}       

\begin{proof}       
Clear from Corollary \ref{cor:2sided}.
\end{proof}       

Terwilliger algebras of quasigroups are only determined up to conjugacy of Cayley tables (Corollary \ref{cor:MCbolisomclasses}), so we use the right inverse property to distinguish potential right Bol loops.

\begin{thm}       
\label{thm:converse}
Let $L$ be the Cayley table of a finite loop $Q$ that has the right inverse property (\ref{eq:BolRIP}).   Suppose that for all $p=(r_p, c_p, r_pc_p)\in X(L)$ and all $c\in Q-\{c_p\}$, both  $\pi_{L, p}^2=\mathrm{id}$  and $\pi(c)=(c_p c^{-1})c_p$.  Then $Q$ is a right Bol loop.
\end{thm}       

\begin{proof}       
Write $c_i=\pi^{i-1}(c)$ and let $r_i$ be such that $r_i c_{p} = r_{p} c_i$.
Now $c_3=c$ since $\pi^2=\mathrm{id}$, so  (\ref{eqn:xiyi-a}) and (\ref{eqn:xiyi-b}) give
\begin{eqnarray}
r_2 c_{p} &=& r_{p} c_2,\label{eq:convrel3}\\
r_2c &=& r_{p}c_{p}.\label{eq:convrel4}
\end{eqnarray}
 Now (\ref{eq:BolRIP}) and (\ref{eq:convrel4}) give
 $r_2=(r_pc_p)c^{-1}$.  Multiplying on the right by $c_p$ gives
 $r_2c_{p}= ((r_pc_p)c^{-1})c_p$. By (\ref{eq:convrel3}), 
 $((r_pc_p)c^{-1})c_p = r_p c_2$.  By assumption $c_2=  (c_p c^{-1})c_p$, so $((r_pc_p)c^{-1})c_p =r_p((c_p c^{-1})c_p)$.  This equation also holds for  $c=c_p$.
 Thus the right Bol axiom holds for all   $r_p$, $c_p$, $c\in Q$.
 \end{proof}        

We expect relatively few Latin squares other than those main class equivalent to the Cayley table of a Bol loop to satisfy $\pi^2=\mathrm{id}$ for all base points (c.f.~Figure \ref{fig:nonexample}).  It is too much to hope that the conditions $\pi^2=\mathrm{id}$ for all base points, the right inverse property, and constant cycle structure  in each column together imply that a loop is right Bol.    

Without the value of $\pi(c)$ in Theorem  \ref{thm:converse}, we only  have the following.

\begin{cor}       
Let $L$ be the Cayley table of a finite loop $Q$ which has the right inverse property (\ref{eq:BolRIP}).  Fix  $p=(r_p, c_p, r_pc_p)\in X(L)$ and $c\in Q-\{c_p\}$.  Let $\pi=\pi_{L,p}$.   Then   $\pi^2(c)=c$ if and only if 
           $((r_p c)c_p^{-1}) \pi(c) = ((r_p \pi(c))c_p^{-1}) c$.
\end{cor}        

\begin{proof}       
Write $c_i=\pi^{i-1}(c)$. 
Then $\pi^2(c)=c$ if and only if $c_3=c$ if and only if $r_2c_3 = r_2c$ if and only $r_pc_p = r_2c$ (by  (\ref{eqn:xiyi-b})) if and only if $r_1c_2 = r_2c$  (by  (\ref{eqn:xiyi-b})).  By the right inverse property,  
$r_1c_2=((r_1c_p)c_p^{-1})c_2$ and 
$r_2c = ((r_2c_p)c_p^{-1})c_1$.  But $r_1c_p=r_pc$ and $r_2c_p = r_pc_2$ by (\ref{eqn:xiyi-a}).  The result follows.
\end{proof}        


\end{document}